\documentclass[a4paper,10 pt]{amsart}

\usepackage{amsmath}
\usepackage{amssymb,amsbsy,amsmath,amsfonts,amssymb,amscd}
\usepackage{latexsym}
\usepackage{graphics}
\usepackage{color}
\usepackage{tikz, tikz-cd}
\usepackage{ulem}
\usepackage{longtable}
\usepackage{xy}
\usepackage{array}
\usepackage{color}
\usepackage{comment}
\input xy
\xyoption{all}

\newcommand\sC{{\mathcal C}}
\newcommand\sT{{\mathcal T}}
\newcommand\sD{{\mathcal D}}
\newcommand\sE{{\mathcal E}}
\newcommand\sA{{\mathcal A}}
\newcommand\sF{{\mathcal F}}

\newcommand\sI{{\mathcal I}}

\newcommand\sB{{\mathcal B}}

\newcommand\sU{{\mathcal U}}

\newcommand\sX{{\mathcal X}}
\newcommand\sY{{\mathcal Y}}

\newcommand\Lam{\Lambda}
\newcommand\al{\alpha}

\newcommand\s{\sigma}

\newcommand\Ga{\Gamma}

\newcommand\ga{\gamma}
\newcommand\de{\delta}

\newcommand{\CC}{\ensuremath{\mathbb{C}}}

\newcommand{\ZZ}{\ensuremath{\mathbb{Z}}}
\newcommand{\QQ}{\ensuremath{\mathbb{Q}}}

\newcommand{\sS}{\ensuremath{\mathcal{S}}}

\newcommand{\NN}{\ensuremath{\mathbb{N}}}
\newcommand{\hol}{\ensuremath{\mathcal{O}}}

\newcommand{\BB}{\ensuremath{\mathbb{B}}}
\newcommand{\PP}{\ensuremath{\mathbb{P}}}

\newcommand{\ra}{\ensuremath{\rightarrow}}

\def\eea{\end{eqnarray*}}
\def\bea{\begin{eqnarray*}}

\newcommand\dual{\mathrel{\raise3pt\hbox{$\underline{\mathrm{\thinspace d
\thinspace}}$}}}
\newcommand\qe{\ifhmode\unskip\nobreak\fi\quad $\Box$}       

\def\BOX{\hfill\lower.5\baselineskip\hbox{$\Box$}}

\newtheorem{theo}[equation]{Theorem}
\newtheorem{remark}[equation]{Remark}

\newtheorem{defin}[equation]{Definition}

\newtheorem{prop}[equation]{Proposition}

\newtheorem{lemma}[equation]{Lemma}
\newtheorem{example}[equation]{Example}

\newcommand{\sR}{\ensuremath{\mathcal{R}}}

\begin{document}

\title[Orbifolds  Quotients of Bounded symmetric Domains]{ Mok tensors and Orbifold Quotients of Bounded symmetric domains without ball factors }
\author{ Fabrizio  Catanese }
\address{Lehrstuhl Mathematik VIII, Mathematisches Institut der Universit\"{a}t
Bayreuth, NW II, Universit\"{a}tsstr. 30,
95447 Bayreuth,}
\email{Fabrizio.Catanese@uni-bayreuth.de}
\author{Marco Franciosi}
 \address{Marco Franciosi\\Dipartimento di Matematica\\Universit\`a di Pisa \\Largo B. Pontecorvo 5\\I-56127  Pisa\\Italy  }
\email{marco.franciosi@unipi.it}

\thanks{AMS Classification:  32Q15, 32Q30, 32Q55, 14K99, 14D99, 20H15, 20K35.\\ 
Key words: Symmetric bounded domains, properly discontinuous group actions, Complex orbifolds, Orbifold fundamental groups, Orbifold classifying spaces.\\ }

\date{\today}

\maketitle

\begin{abstract}
 In this paper we 
characterize  the compact orbifolds, quotients $ X = \sD/ \Ga$
of a bounded symmetric domain $\sD$ with no higher dimensional ball factor by the  action of a discontinuous  group $\Ga$,
as those projective orbifolds with ample orbifold canonical divisor which admit  a Mok curvature type tensor of orbifold type 
and satisfying
certain  other conditions implying the existence of a finite smooth covering.
\end{abstract}


\section*{Introduction}

The main purpose of this article is to extend the description of compact quotients of bounded symmetric domains $
\sD$, given in \cite{cat25} for a bounded symmetric domain $\sD$ of tube type, to the more general case
where $\sD$ contains no ball factors of dimension $\geq 2$. 

The general framework is as follows: if  $M$ is a simply connected complex manifold (as $\sD$ above),  and $\Ga$ is
a properly discontinuous group of biholomorphic automorphisms of $M$,  we consider 
the quotient normal complex space $ X = M / \Ga$, and we assume that  $X$ is compact.

 In the case where the set $\Sigma$ of points $z \in M$ whose stabilizer is nontrivial has codimension $1$, we
 add some information taking into account the irreducible Weil divisors
 $D_i \subset X$,
whose union is the codimension $1$ part of the branch locus $\sB : = p (\Sigma)$ of  $ p : M \ra X$.

 We obtain in this way a  global orbifold $\sX$, consisting of the pair   $(X, D)$, where $D$ is the $\QQ$-divisor 
$D:= \sum_{i\in I} (1-\frac{1}{m_i})D_i$, where  
for   each  divisor $D_i$ the   integer $m_i$ is  equal to  the order of the stabilizer group 
at a general point of the inverse image of $D_i$.

The above global orbifold $\sX = M / \Ga$ is said to be {\bf good} if  $\Ga$ admits a finite index normal subgroup 
 $\Lambda \triangleleft \Ga$ which acts freely (equivalently,   $\Lambda$  is torsion free). Then we denote  by  $Y$ 
the compact complex manifold $ Y : = M / \Lambda$,
and  we have then $ X = Y/G$, where $G$ is the finite quotient group $G : = \Ga / \Lambda$. 

If  $M$ is  contractible, 
 $Y$ is a classifying space $K(\Lambda, 1)$ for the group $\Lambda$, and $\sX$ is an orbifold classifying
space for the orbifold fundamental group $\Ga$ of $\sX$  (see \cite{cat24}, \cite{cat25}
for these definitions).

The case where $M = \CC^n$,  was considered in \cite{cat24}, the case where $M = \sD \subset \subset \CC^n$
 is a bounded symmetric domain of tube type was treated in \cite{cat25}
(see \cite{Koba59}, \cite{Helgason2} and also \cite{MokLibro}, \cite{CaDS}). 

For $M = \sD$ a bounded symmetric domain we have a
good global orbifold, by virtue of the so-called Selberg's Lemma (\cite{selberg}, \cite{cassels}). 

 In the case where $\Ga$ acts freely, we have some explicit criteria, stated in  \cite{CaDS}, \cite{CaDS2}:
 these are the basis for our results in the case where $\Ga$ does not act freely (but with compact quotient),
 and we combine them  with the ideas of  \cite{cat24} and  \cite{cat25},  using the language of  orbifolds
 to state our main results.


In this paper, given   a compact  orbifold $\sX = (X,D)$ where  $X$ is  an $n$-dimensional normal variety  and $D= \sum_{i\in I} (1-\frac{1}{m_i})D_i$,
(with  $m_i$ a strictly positive  integer)
we will investigate under which conditions the orbifold universal covering  is  a bounded symmetric  domain $\sD$
having no  irreducible factor $\sD_j$ 
 isomorphic to  a ball $\BB^r : = \{ z \in \CC^r :  |z|<1\}$ with   $r \geq 2$,  in its irreducible decomposition  $\sD = \prod_j \sD_j$. 

 Necessarily, a global  orbifold quotient is a Deligne-Mostow orbifold $\sX$. Hence 
we will make  first  the assumption that  $\sX$ is
a Deligne-Mostow orbifold (see \cite{d-m} Section 14), that is, locally modelled as the quotient of a smooth manifold
 by a finite group (hence $D$ is $\QQ$-Cartier).
 
 Later on, in order to relate to questions of classification theory,   we shall make the weaker
 assumption that we have a klt orbifold,  i.e., we will consider 
the case where the pair  $(X,D)$ is  klt (Kawamata log terminal).

Another important observation is that 
 necessarily   $X$ must be projective, since  indeed,  by \cite{Kodaira},   the orbifold canonical divisor $ K_{\sX} := K_X + D $ must be ample (this condition may be  weakened later on).

   Our study is based on a generalization to the orbifold case of  the ideas of \cite{CaDS}, \cite{CaDS2}, 
   relying on  the tensors introduced by Mok in  \cite{MokLibro}, \cite{Mok1},  as developed in    \cite{cat24} and  \cite{cat25},  
and 
  key instruments for our analysis are 
 the orbifold fundamental group  $\pi_1^{orb}(\sX)$, the notion of orbifold covering, and   the existence of a K\"{a}hler-Einstein orbifold metric
 as shown in \cite{Eys}.

 Our first  result is the following:

 \begin{theo}\label{easier}
 
  The global compact complex orbifolds $\sX=(X,D)$ of bounded symmetric  domains $\sD$
with no  irreducible factors isomorphic to  a higher dimensional ball  
   are the projective complex orbifolds such that:

  (1)  $ K_{\sX} := K_X + D= K_X + \sum_i \frac{m_i-1}{m_i}  D_i$ is ample;

  (2) $\sX$ admits      a  Mok curvature-type holomorphic tensor (see Def. \ref{curvature} and \ref{mok})
  $$\s_X \in    \Omega_X^1(\log (\lceil D \rceil))\otimes \Omega_X^1(\log (\lceil D \rceil)) \otimes T_X \otimes   T_X,$$
which is of orbifold type (see Def. \ref{orbifold type} and Lemma \ref{descent});

(3)   the orbifold fundamental group $\Ga$ admits a torsion free normal
subgroup $\Lambda$ of  finite index such that,  setting $G : = \Ga / \Lambda$,  
 the corresponding Galois covering $ Y \ra X= Y/G$ yields an 
orbifold $\sY$ which is reduced and

(4)  either

(4-i)  $\sY$ is a  smooth
orbifold,

or

(4-i')  $Y$ has singularities which are 2-homologically connected,
that is, they have a resolution of singularities $\pi : \tilde{Y} \ra Y$ such that $\sR^j\pi_* (\ZZ_{\tilde{Y} } )=0$
for $j=1,2$.

  \end{theo}

 We can slightly change some  of the assumptions and obtain the  following:

  \begin{theo}\label{dm}
 
  The global compact complex orbifolds $\sX$ of bounded symmetric  domains $\sD$
without irreducible factors isomorphic to  a higher dimensional ball  
   are the projective complex orbifolds such that:  
   
     ($\sA$-1) $\sX$ is a Deligne-Mostow orbifold;
     
      hypotheses (1), (2), (3)  are verified.

 \end{theo}

  \begin{theo}\label{dm2}
 
  The global compact complex orbifolds $\sX$ of bounded symmetric  domains $\sD$
with no irreducible factors isomorphic to  a higher dimensional ball  
   are the projective complex orbifolds such that:  
   
    ($\sA$-2) $\sX$ is a Deligne-Mostow orbifold, with residually finite orbifold fundamental group $\Ga$, 
    which  is an orbifold classifying space  with  reduced and smooth orbifold universal covering $\tilde{X}$;
   
    hypotheses (1), (2) are verified.

 \end{theo}

Concerning the klt case, using the  works   \cite{GKP16}, \cite{GKPT}, \cite{GGK19},
 we are able to obtain an improvement of Theorem \ref{dm}:

  \begin{theo}\label{theorem klt} 
  
  Let $\sX=(X,D)$ be a  klt orbifold.  
  
  Then the orbifold universal covering  $\pi: \tilde{\sX} \to \sX$ is  given by
 a bounded symmetric  domain  $\sD$
with no irreducible factors isomorphic to  a higher dimensional ball  
if and only if  conditions (1), (2), (3) hold. 

 \end{theo}
 
 Theorem \ref{theorem klt} in the reduced case answers  question (1) posed in \cite{CaDS}, showing that, for $X$ smooth, the hypothesis
 $K_X$ ample can be replaced by $K_X$ big and nef (since we may then consider a minimal model of $X$ with klt singularities).
  
We do not give a fully self contained proof of Theorem \ref{theorem klt}:   one   argument needed in the reduced case has recently  appeared
  in  \cite{GrPa2},  hence we
   refer to   Proposition 6.3 loc. cit., 
   showing that  the existence of a Mok tensor implies  the existence of a uniformizing VHS.

  \section{Complex orbifolds, Deligne-Mostow orbifolds, orbifold fundamental groups, orbifold coverings} \label{section orbifold}

\subsection{Notation}
Throughout  this paper $X$ is an
 $n$-dimensional (irreducible) normal compact complex space. 
 \begin{itemize}
 \item
   $ Sing(X)$ is   the singular locus of $X$,   $X_{reg}$ is the smooth locus $X_{reg}: = X \setminus Sing(X)$.
\item
An orbifold  $\sX$ is a pair $(X,D)$ where  $X$ is  an $n$-dimensional normal  complex space 
   and $D= \sum_{i\in I} (1-\frac{1}{m_i})D_i$   where the     $D_i$'s  are  distinct irreducible Weil divisors,  and 
   $m_i \in \NN$, $m_i\geq 2$.  
\item
The integer $m_i$ is called the {\bf order}  of $D_i$.
\item
$ \lceil D \rceil $ or $supp(D)$ is the  reduced divisor $ \sum D_i$, 
 and  $D_{sing}$ is the singular locus of 
$ \lceil D \rceil $. 
\item
$\sX$  is  {\bf quasi-smooth}  if moreover $D_i$ is smooth outside of $Sing(X)$.
\item
$\sX$  is   {\bf reduced}  if  $supp(D)=\emptyset$.
\item
 $\sX$  is a {\bf Deligne-Mostow orbifold} if  for each point $x \in X$ there exists an open set $ U $
 containing $x$ such that $\sX$ restricted to $U$ is isomorphic to an orbifold 
  $ \Omega/G$, where $ \Omega \subset \CC^n$.
  \item
  The underlying complex space $X$ of a D-M (= Deligne-Mostow) orbifold  has only quotient singularities (these are rational singularities).
  \item
  If $\sX = M / \Ga$ is the quotient of a complex manifold $M$ by a properly discontinuous subgroup $\Ga$, then
we say that $\sX$ is a {\bf global}  D-M orbifold.
\item
The {\bf orbifold canonical divisor} is defined as 
 $$ K_{\sX} := K_X +D = K_X +\sum_i  (1- \frac{1}{m_i})  D_i.$$
 \item
 The {\bf orbifold fundamental group} $\pi_1^{orb} (\sX)$ is defined as the quotient 
$$ \pi_1^{orb} (\sX) : = \pi_1 (X_{reg} \setminus \lceil D \rceil ) / \langle \langle(\ga_1^{m_1}, \dots, \ga_r^{m_r} , \dots \rangle \rangle$$
of $\pi_1(X_{reg} \setminus \lceil D \rceil )$ by the subgroup normally generated by the elements $\ga_i^{m_i}$, where 
the  $\ga_i$'s are simple geometric loops
  going each around a smooth point of the divisor $D_i$ (and counterclockwise). 

 \end{itemize}

\subsection{Orbifold coverings} 
If $\sX$ is reduced (i.e. $supp (D)=\emptyset$) 
the orbifold fundamental group
is just  the fundamental group of $X_{reg}$, and its subgroups correspond to finite coverings which are unramified on the smooth locus, thanks to the extension by Grauert and Remmert \cite{g-r} of Riemann's existence theorem to finite holomorphic maps
of normal complex spaces.

To subgroups of the orbifold fundamental group correspond instead the so called 
 {\bf orbifold covering} of orbifolds.

\begin{defin}[Orbifold covering (see for instance  \cite{d-m})]\label{def: orbifold covering}
Let  $\sX=(X,D)$ be an orbifold,   where  
$D= \sum_{i\in I} (1-\frac{1}{m_i})D_i$. 

An orbifold covering  $f\colon \sY=(Y, D')  \to \sX=(X,D) $ is a holomorphic  map  where   $(Y,D')$  is a pair with 
$D' = \sum_{i,j} ( 1- \frac{1}{n_{i,j}})D_{i,j}'$,   satisfying the following conditions:  

\begin{itemize}  
\item 
 every $x \in X$ has a connected neighborhood $V \subset X$ such that every connected component
 $U$ of $f^{-1}(V)$ meets the fibre $f^{-1}(x)$ in only one point, and $f_{|U }\colon  U \to V$ is finite;  
\item 
$f$ induces an unramified map $$ F \colon Y \setminus f^{-1}(X_{sing}\cup  \lceil D \rceil ) \ra X_{reg}  \setminus \lceil D \rceil ;$$ 
\item $ \forall i  \in I$,
$f^{-1}(D_i) $ is set theoretically a union of divisors $D_{i,j}'$  such that 
\begin{enumerate}  

\item  $D_{i,j}' $ maps onto $D_i$,  
\item  there is an integer  $r_{i,j}$  with $m_i = r_{i,j} n_{i,j}$ such that the multiplicity of $D_{i,j}$ in
$f^* (D_i) $ equals $ r_{i,j}$,  while the  multiplicities of the other components of $f^* (D_i) $  are equal to $m_i$.

\end{enumerate} 
\end{itemize}

\end{defin}

We recall from \cite{cat24},  \cite{d-m}:

\begin{itemize}
\item
There exists a natural one-to-one correspondence between subgroups $\Ga <  \pi_1^{orb}(X, D)$ and 
orbifold coverings $f\colon \sY \to \sX$. 
\item
The orbifold universal cover 
$ \tilde{\sX} := ( \tilde{X}, \tilde{D})$ is the one corresponding to the trivial subgroup.
\item
 An  orbifold $\sX = (X, D)$ is said to be an {\bf orbifold classifying space} if its 
universal covering $\tilde{\sX} = ( \tilde{X}, \tilde{D} )$
admits  a homotopy retraction of  to a point which preserves the subdivisor $\tilde{D}$.
\item
The orbifold $\sX$ is said to be klt (Kawamata log terminal) if the pair $(X,D)$ is klt.
\item
Klt pairs behave well under orbifold coverings $  ( Y, D') \to (X,D) $:  
by \cite[Cor. 2.43]{Kol} , the pair $(X,D)$ is klt if and only the pair  $ (Y,D')$ is klt. 
\end{itemize}

\subsection{Local orbifold fundamental group}
\begin{defin}
Consider  an  orbifold $(X, D)$ and a point $p \in X $.

Choose a small good open neighbourhood of $p$ such that there exists a homotopy  retraction 
of $U$
to $p$, induced by $ U \times [0,1]\ra U$  preserving the divisors $D_i\cap U : = D_i|U$.

 Then we get an induced orbifold $\sU : = (U, D_{|U})$ and a homomorphism of orbifold fundamental groups 
$$ \iota_p: \pi_1^{orb} (U,  D_{|U})\ra  \pi_1^{orb} (X , D ) . $$ 
$ \pi_1^{orb} (U ,  D_{|U}) $ is called the {\bf local orbifold fundamental group at p.} 

It is trivial for $p \in X \setminus D$.
\end{defin}

\begin{prop}
(1)   $\sX$ is Deligne-Mostow at $p$
if and only if the local orbifold fundamental group at $p$ is finite and the orbifold local universal covering $\tilde{\sU}$ is a reduced and smooth orbifold. 

(2) The orbifold  universal covering $\tilde{\sX}$ is reduced and smooth if  
and only if for each $p$ the local orbifold  universal covering $\tilde{\sU}$ 
is reduced and smooth and  $\iota_p$ is injective.
\end{prop}

\begin{proof}
(1): the conditions are clearly necessary.

Conversely, let $G$ be the local orbifold fundamental group, and assume that $\tilde{\sU}$ is a smooth manifold $M$.

Then $U = M / G$, and, shrinking $U$, there is $U' \subset \CC^n$ such that  $U = U'   / G'$, where $G'$ is the isotropy group of a point $p'$ lying above $p$.

(2) The if assertion  follows   since then $\tilde{\sX} |U$ induces a trivial covering of $\tilde{\sU}$.

Concerning the only if assertions of  (2),
if    $\tilde{\sX}$ is reduced and smooth, then clearly also $\tilde{\sU}$ 
is reduced and smooth.

 If  $\iota_p$ is not injective then first of all necessarily $ p \in D_i$ for some $i$.

   $\tilde{\sX} |U $ is an orbifold  covering of $\sU$, which  is    reduced and smooth.
   
   If $p$ is a smooth point of $X$ and $p$ belongs to only one
codimension $1$ component $D_i$ of $D$, smooth at $p$,  then the local fundamental group
is generated by $\ga_i$, hence  $\iota_p$ is  injective.

  More generally,  $\tilde{\sX} |U $, being an   orbifold  covering of $\sU$, is such that its connected components admit as
   orbifold coverings  the universal 
   covering   $\tilde{\sU} \to \sU $.

To the kernel $\Ga'$ of $\iota_p$ corresponds  a connected  orbifold
 covering $\sU'\to \sU$ and we have a factorization
$$\tilde{\sU}  \to \sU' \to \sU$$

The connected components of   $\tilde{\sX} |U $ are locally isomorphic to $ \sU' $,
hence $ \sU' $ is smooth and reduced; it follows then  that   $\tilde{\sU}$, which is an unramified covering 
of $ \sU' $, is also smooth and reduced.

\end{proof}

\section{ Algebraic curvature-type tensors, their First Mok characteristic varieties, and Mok tensors}

We recall in this section the notation and the results of \cite{MokLibro},  \cite{Mok1},  and \cite{CaDS2},
which studied and applied the curvature tensor $\s$ introduced by Kobayashi and Ochiai in \cite{KobOchi}
for locally symmetric manifolds.

\begin{defin}\label{curvature}
Given a direct sum
$$T = T_1 \oplus  ... \oplus T_k$$
of irreducible representations $T_i$ of a group $H_i$,

1)
An algebraic {\bf curvature-type tensor} is a nonzero  endomorphism
$$ \s   \in End (T \otimes T^{\vee}).$$
2) Its first Mok characteristic cone  $\sC\sS \subset  T$ is defined as the projection on the first factor
of the intersection of $ker (\s)$ with the set of rank 1 tensors, to which we add  the origin:
$$ \sC\sS : = \{ t \in T | \exists t^{\vee} \in T^{\vee} \setminus \{0\}, (t \otimes  t^{\vee} ) \in ker (\s)  \}.$$
3) Its {\bf first Mok characteristic variety} is the associated subset $\sS : = \PP ( \sC\sS) \subset \PP (T)$.

4) More generally, for each integer $h$, to 
 the algebraic cone 
 $$ \sC\sS^h : = \{ t \in T | \exists A  \in ker (\s), {\rm Rank} (A) \leq h  , \exists t'  \in T: t = A \llcorner t'  \},$$ 
 is associated  the {\bf h-th Mok characteristic variety} $\sS^h : = \PP ( \sC\sS^h) \subset \PP (T)$.

The  full characteristic sequence is then the sequence
$$ \sS = \sS^1   \subset \sS^2 \subset \dots \subset \sS^{k-1} \subset \sS^k =  \PP (T).$$

\end{defin}

\begin{remark}\label{Mok}
\begin{enumerate}
\item
In the case where $\s$ is the curvature tensor of an irreducible symmetric bounded domain $\sD$, 
 $T = T_1 \oplus  ... \oplus T_k$ is the tangent space at one point and $H_1 \oplus  ... \oplus H_k$
is the restricted holonomy group, and 
Mok (\cite{Mok1})
proved that the difference sets  $ \sS^h \setminus    \sS^{h-1}$ are exactly all the  orbits of the  parabolic subgroup $P$
 associated to the compact dual
$\sD^{\vee} = G/P$.  In particular, the algebraic cone  $ \sC \sS^h$ is irreducible and $H_i$-invariant.
\item
More generally, if $ \sS \subset T_i$ is an $H_i$-invariant algebraic cone, and  $H_i$ is the holonomy group of an  irreducible symmetric bounded domain,
then necessarily $ \sS$ is irreducible and indeed equal to one of the $ \sS^h$.
\item
In the case instead where $H_i$ acts as the full unitary group on $T_i$, then any $H_i$-invariant algebraic cone in $T_i$ is trivial, that is, either equal to
$T_i$ or just equal to  $\{ 0 \}$, where $0 \in T_i$ is the origin.

\end{enumerate}

\end{remark}
\begin{defin}\label{mok}
  A holomorphic curvature type tensor 
  $ \s   \in End (T_M \otimes T_M^{\vee})$
  on a connected complex manifold $M$ 
  is called a {\bf Mok curvature-type holomorphic tensor}, or, more briefly, a {\bf Mok tensor} if 
there exists  a   point  $p \in M$, and  a splitting of the tangent space $T = T_{M,p} $
$$T = T'_1 \oplus  ... \oplus T'_m$$
such that

(A.1)  the first Mok characteristic cone  $\sC\sS$  of $\s_p$  is $\neq T$ 

(A.2)  $\sC\sS$ splits into m
irreducible components $\sC\sS' (j )$  with

B) $ \sC\sS' (j )  =  T'_1 \times  ... \times  \sC\sS'_j \times    ...   \times T'_m$

C)  $\sC\sS'_j  \subset T'_j $  is the cone over a smooth  non-degenerate projective variety $\sS'_j$ (non-degenerate means that  the cone $\sC\sS'_j $ 
spans the vector space $T'_j$) 
unless  $ \sC\sS'_j = 0$ and dim $(T'_j) = 1.$

(II) If $X$ is singular, a meromorphic curvature type tensor $\s$ is said to be a Mok tensor
if there is a smooth point $p$ of $X$ where $\s$ is holomorphic and where properties (A.1), (A.2), B), C) hold.

\end{defin}

In the case where there are no ball factors,  situation (3) does not occur, 
and Theorem \ref{easier} is a generalization of the following  \cite[Thm. 1.2]{CaDS2}.

\smallskip
\begin{theo}\label{manifold}
Let $X$ be a compact complex manifold of dimension $n$ with $K_X$  ample.

Then the universal covering $\tilde{X}$ is a bounded symmetric domain with no    higher  ball factors
if and only if there is a   Mok curvature type holomorphic tensor  $\s \in H^0 ( End (T_X \otimes T_X^{\vee}))$.

Moreover, we can recover the universal covering  of $\tilde{X}$ from the sequence of pairs
$ ( dim (\sC\sS'_j ) , dim (T'_j))$.

\end{theo}

The above result can be  simplified if one restricts to locally symmetric varieties $X$ whose universal covering
$\tilde{X}$ is a bounded symmetric domain without factors of rank one (see loc. cit.), but we shall not write up here
 the generalization of this result to the orbifold case. 

\smallskip

\section{ Symmetric domains and descent of  curvature type  tensors}

 \subsection{Curvature type tensors and their descent}  
 
 \hfill\break
 
 Every orbifold covering around a smooth point of  X,  lying on the smooth locus of an orbifold divisor 
 $D_i$ having order $m_i$,
   can be seen locally as a  ramified covering  
$$ \begin{array} {rll} Y:=\CC^n & \to  & \CC^n =X  \\
(y_1, \dots, y_{n-1}, y_n) & \mapsto & (x_1=y_1, \cdots, y_{n-1}=x_{n-1}, x_n=y_n^{q}) \end{array} $$
where  $D_i= (x_n=0)$ and $ q | m_i$.
 
If we let $G$ be the local Galois group of the covering, we want to see a relation between the condition that $\s$ 
is a  $G$-invariant  algebraic curvature-type holomorphic tensor
 $\s  \in End \big[  T_Y \otimes  \Omega_Y^1] $
 and its property of descending to an   algebraic curvature-type meromorphic tensor
 on $X$. 
 
  Let $\sX=(X,D)$ and let $ \lceil D \rceil $ be  the  reduced divisor $ \sum D_i$. 
 
 Recall that the sheaf $ \Omega_X^1(\log (\lceil D \rceil))$ is the sheaf generated by $ \Omega_X^1$ and by the differentials
 $ d log (F_i)$, where the $D_i = div(F_i)$'s  are the components of $D$; its dual is denoted by  $T_X(-\log (\lceil D \rceil)) $
 and, as shown  in \cite{cimemoduli}, it can be seen as  the sheaf of holomorphic tangent vectors  carrying each   ideal $\sI_{D_i}$
 to itself.

Ideally, one would like to show that  
$\s$ descends to some holomorphic tensor $\hat{\s}_X \in End\big[ T_X(-\log (\lceil D \rceil))  \otimes  \Omega_X^1(\log (\lceil D \rceil))\big]  $ on the smooth locus
$X_{reg}$ of $X$. The following Lemma shows however that this is only a meromorphic such tensor  in the case $q=2$. At any rate, one 
then  defines 
$$ \s_X : = i_* ( \hat{\s}_X), \ \
{\rm for }  \ i : X_{reg} \ra X  \ {\rm being \ the \ inclusion}.$$

\begin{lemma}\label{descent}
Consider the action of the cyclic group $$G= \mu_q : = \{ \zeta| \zeta^q=1\}$$ of qth roots of unity at the origin in $Y : = \CC^n$, via the linear action,
for $\zeta \in \mu_q$:
$$ (y_1,\cdots,y_{n-1},y_n)  \mapsto(y_1,\cdots,y_{n-1}, \zeta y_n) $$
whose  quotient map is $\pi : Y \ra X : = \CC^n$,
$$\pi : (y_1, \dots, y_{n-1}, y_n) \mapsto (x_1=y_1, \cdots, x_{n-1}=y_{n-1}, x_n=y_n^q). $$ 

Let $D\subset X$ be the branch divisor, defined by $x_n=0$, and let $D'\subset Y$ be the ramification divisor, defined by $y_n=0$.

Then a  $G$-invariant  algebraic curvature-type holomorphic tensor $ \s \in  End (T_Y \otimes  \Omega_Y^1 ))$
descends, for $q \geq 3$,  to a  holomorphic   tensor  
$$\s_X \in End \big[ T_X(-\log (\lceil D \rceil))  \otimes  \Omega_X^1(\log (\lceil D \rceil))\big] \subset  \Omega_X^1(\log (\lceil D \rceil))\otimes \Omega_X^1(\log (\lceil D \rceil)) \otimes T_X \otimes   T_X , $$
and for $q=2$ to a holomorphic tensor 
$$\s_X \in    \Omega_X^1(\log (\lceil D \rceil))\otimes \Omega_X^1(\log (\lceil D \rceil)) \otimes T_X \otimes   T_X.$$

Conversely, setting $\xi_i =  \partial / \partial  x_i, \ i=1, \cdots, n,$  $f_i = log(x_i)$, 
 a  holomorphic   tensor  
 $$\s_X \in    \Omega_X^1(\log (\lceil D \rceil))\otimes \Omega_X^1(\log (\lceil D \rceil)) \otimes T_X \otimes   T_X,$$
 written as 
$$ \s_X =  \sum_{i,j,h,k=1}^n  B'''_{ijhk}(x) df_i \otimes df_j \otimes \xi_h \otimes \xi_k$$
lifts to  a  holomorphic   tensor  $\s$ on $Y$ if and only if  $B'''_{ijhk}(x)$ is divisible by $x_n^{t + p(s,t)} $ 
where, denoting by  $[m] \in \{0,\dots, q-1\}$  the residue class of $m$, 
$$  s: =  |\{ i,j =n\} | , \ t : =   |\{h,k=n\}|, \ \ p(s,t) : = \frac{s-t + [t-s]} {q}  .$$

\end{lemma} 

\begin{proof}

Let us see how  local sections  of  $T_Y$  (resp. $\Omega_Y^1$)  descend to $X$.  


Let  $\eta_i := \partial / \partial  y_i, \ i=1, \cdots, n,$  be the standard  local frame  for $T_Y$, respectively  
$\xi_i =  \partial / \partial  x_i, \ i=1, \cdots, n,$ be the standard  local frame  for  $T_X$. 

Write $\s$,  a holomorphic tensor $ \s \in  End (T_Y \otimes  \Omega_Y^1 ))$,
as a holomorphic section of 
$$ (T_Y \otimes  \Omega_Y^1)^{\vee} \otimes (T_Y \otimes  \Omega_Y^1 )
\cong  (   \Omega_Y^1 \otimes T_Y) \otimes (T_Y \otimes  \Omega_Y^1) 
\cong     \Omega_Y^1 \otimes \Omega_Y^1 \otimes T_Y \otimes T_Y  .  $$

 Then we write $\s$ as 
 \medskip
  
$\s= \sum_{i,j=1}^n  \sum_{h,k=1}^n a_{ijhk} (y_1,\cdots,y_{n})dy_i \otimes dy_j \otimes \eta_h \otimes \eta_k=  $

 \medskip
 
$ = \sum_{i,j,h,k=1}^n (\sum_{m \geq 0} A_{ijhkm}(y_1,\cdots,y_{n-1})y_n^m)dy_i \otimes dy_j \otimes \eta_h \otimes \eta_k.$ 
 \medskip
 
The action of $G$  leaves the frame tensors $dy_i \otimes dy_j \otimes \eta_h \otimes \eta_k$ invariant, except if the
index $n$ occurs: and $dy_n \mapsto \zeta dy_n, \eta_n  \mapsto \zeta^{-1} \eta_n$.

Hence the action of $\zeta$ leaves $\s$ invariant if and only if we have only terms $A_{ijhkm}$
where $$ m + |\{ i,j =n\} |- |\{h,k=n\}|\equiv 0 \ mod (q).$$

Letting $ s : =  |\{ i,j =n\} |, t : =  |\{h,k=n\}| \Rightarrow  m \equiv t - s  \ mod (q),$ 
this means that $$A_{ijhk} (y) = B_{ijhk} (x) y_n ^{[t-s]},$$
where $[m] \in \{0,\dots, q-1\}$ denotes the residue class of $m$.

Now,  for $i \leq n-1$,  we may identify $\eta_i$ with $\xi_i$ and $dy_i $ with $dx_i$.
  
 Whereas, from  
the relation $x_n=y_n^q$ we get $$ q \ d \log (y_n) = d \log (x_n) \Leftrightarrow  q d y_n  = y_n \frac{d x_n }{x_n}  ,$$ 

and  $$  y_n  \eta_n = q x_n \xi_n  \Leftrightarrow \eta_n = q \frac{x_n}{y_n} \xi_n.$$

The left hand sides of these formulae mean that the standard frame  $df_1, \dots, df_n$ for $ \Omega_X^1(\log (\lceil D \rceil))$
 is a frame for 
$ \Omega_Y^1(\log (D'))$, where $D'\subset Y$ is   defined by $(y_n=0)$,  and the standard frame $\de_1, \dots, \de_n$ 
 for $T_X(-\log (\lceil D \rceil))$ is a frame for $T_Y(-\log (D'))$.

Whence if $\s$ is $G$-invariant, 

\medskip

$
 \s = \sum_{i,j,h,k=1}^n  A_{ijhk}(y)dy_i \otimes dy_j \otimes \eta_h \otimes \eta_k=$
 
 $
 =  \sum_{i,j,h,k=1}^n  B_{ijhk}(x) y_n ^{[t-s]} dy_i \otimes dy_j \otimes \eta_h \otimes \eta_k= $
 
 $
 =  \sum_{i,j,h,k=1}^n  B'_{ijhk}(x) y_n ^{s -t + [t-s]} df_i \otimes df_j \otimes \de_h \otimes \de_k.$

 \medskip
 
Observe now that $s-t \in \{ -2, -1,0,1,2\}$, hence $ [t-s] - (t-s)$, which is divisible by $q$, yields a power of $x_n$ with exponent 
$p(s,t)$, where, since $- t \leq  [t-s] - (t-s) \leq q + 1 $, $p(s,t)=0$ or $p(s,t)=1$ except if $t=2$, $s=0$, $q=2$,
in which case $p(s,t)=-1$.

We have therefore proven that $ \s \in  (End (T_Y \otimes  \Omega_Y^1 ))^G$
descends to a  holomorphic   tensor  
$$\s_X \in    \Omega_X^1(\log (\lceil D \rceil))\otimes \Omega_X^1(\log (\lceil D \rceil)) \otimes T_X \otimes   T_X,$$
 $$ \s_X =  \sum_{i,j,h,k=1}^n  B'_{ijhk}(x) x_n ^{p(s,t)} df_i \otimes df_j \otimes \de_h \otimes \de_k=$$
 $$ =  \sum_{i,j,h,k=1}^n  B'_{ijhk}(x) x_n^{t + p(s,t)} df_i \otimes df_j \otimes \xi_h \otimes \xi_k.$$

And, for $ q \geq 3$, 
$$\s_X \in   \Omega_X^1(\log (\lceil D \rceil))\otimes \Omega_X^1(\log (\lceil D \rceil))  \otimes  T_X(-\log (\lceil D \rceil))  \otimes   T_X(-\log (\lceil D \rceil)) \cong $$
$$ \cong  End \big[ T_X(-\log (\lceil D \rceil))  \otimes  \Omega_X^1(\log (\lceil D \rceil))\big] . $$

Moreover, for $ q \geq 3$, a holomorphic tensor $\s_X \in End \big[ T_X(-\log (\lceil D \rceil))  \otimes  \Omega_X^1(\log (\lceil D \rceil))\big]  $,
$$ \s_X =  \sum_{i,j,h,k=1}^n  B''_{ijhk}(x) df_i \otimes df_j \otimes \de_h \otimes \de_k$$
lifts to a holomorphic tensor $\s$ on $Y$ if and only if $B''_{ijhk}(x)$ is divisible by $x_n ^{p(s,t)}$
when ${p(s,t)}=1$, a condition which is equivalent to $ s=t +1$.

Hence, conversely,  including all the cases  $q\geq 2$,  a  holomorphic   tensor  
$\s_X \in  \Omega_X^1(\log (\lceil D \rceil))\otimes \Omega_X^1(\log (\lceil D \rceil)) \otimes T_X \otimes   T_X,$ written as 
$$ \s_X =  \sum_{i,j,h,k=1}^n  B'''_{ijhk}(x) df_i \otimes df_j \otimes \xi_h \otimes \xi_k$$
lifts to  a  holomorphic   tensor  $\s$ on $Y$ if and only if  $B'''_{ijhk}(x)$ is divisible by $x_n ^{t + p(s,t)}$,
where  $ 0 \leq t + p(s,t) \leq 3$.

\end{proof}

The above Lemma allows us to  introduce the notion of a tensor $s_X$  of orbifold type.

\begin{defin}\label{orbifold type} 
We shall say that,  on an orbifold $\sX = (X, D)$, a  holomorphic tensor 
$$\s_X \in    \Omega_X^1(\log (\lceil D \rceil))\otimes \Omega_X^1(\log (\lceil D \rceil)) \otimes T_X \otimes   T_X,$$
 is of {\bf orbifold type} if  its pull back  
  to   any reduced orbifold covering $\sY  \to \sX$ (i.e.,  branched exactly at $D$ with the appropriate multiplicities)  is holomorphic. 
 
 Equivalently, by  Lemma \ref{descent}, this means that if  $D_i$ is an orbifold irreducible divisor with order  $m_i$,
 then in the local expansion around $D_i = \{x_n=0\}$, the coefficients $B'''_{ijhk}(x)$ are divisible by $x_n^{t + p(s,t)} $
 (where we set $q=m_i$). 
\end{defin}

 \section{Proof of the main Theorems}

\subsection{Necessary conditions.  }

 For the necessary conditions,  observe that if 
$\sD$  is a   bounded symmetric domain and 
 $\sX$ is  a compact global orbifold  associated to a properly discontinuous subgroup $\Ga$ of biholomorphisms of $\sD$, 
then, by Selberg's Lemma,
the group $\Ga$ admits a finite index torsion free normal subgroup $\Lambda$   and  we obtain a compact manifold  
$ Y : = \sD / \Lambda$ such that   $  X = Y/G$, where $G$ is the finite quotient group $G : = \Ga / \Lambda$: hence (3) holds.
 
Considering  $\sB:  =\sum D_i$ the branch locus of $\pi \colon Y \to X$  and letting $m_i$ being the branching multiplicities,
we obtain an orbifold structure $\sX= (X,D)$ where $D= \sum (1- \frac{1}{m_i} ) D_i$   and an orbifold covering $ Y \to (X,D)$. 
 
Note that  $\sX $ is a Deligne-Mostow orbifold and the  singularities of $X$ are quotient singularities,
since, at any point $ y \in Y$ having a nontrivial stabilizer $G_y < G$, the group $G_y$ acts linearly 
by Cartan's Lemma \cite{cartan}. 
By 
 \cite{k-m} (Prop. 5.15, page 158)  quotient singularities $(X,x)$ are rational singularities, that is, they are normal and,
 if $\varphi : Z \ra X$ is a local resolution, 
then $\sR^i\varphi_* \hol_Z=0 $ for $i \geq 1$.

 Since $\sD$ is contractible, $\sX$ is an orbifold classifying space.

The canonical divisor $K_Y$ is ample  
by the classical results of   Siegel  (\cite[Thm. 3]{Sie73}) and
Kodaira (\cite[Thm. 6]{Kodaira}),
 and  $Y$ is projective. 

Moreover 
$Y$ is projective if and only if  $X$ is projective since, by averaging,  we can find on $Y$ a $G$-invariant very ample divisor. Since 
 $$ K_Y = \pi^{\ast} (K_{\sX} )=\pi^{\ast} (K_X +\sum_i  ( 1 - \frac{1}{m_i}  D_i)),$$
$K_{\sX}$ is ample, and  condition  (1) is verified.

The condition on the tensors and on the splitting of the tangent bundle,
follows from Definition \ref{orbifold type} and Lemma \ref{descent} and from  Theorem \ref{manifold}.

For Theorem \ref{dm2}, observe that $\Ga$ is a linear group, hence it is residually finite.

Finally, for Theorem \ref{theorem klt} if $\sX=\sD/\Gamma$, 
 we have a compact manifold  
$ Y : = \sD / \Lambda$  and a Galois morphism $Y \to (X, D)$.  
The  pair $(X,D)$ is klt    
by 
   \cite[Cor. 2.43]{Kol} since  $(Y, 0)$ is  klt, because $Y$ is smooth.

 \subsection{Sufficient conditions.} 
 \subsection{Proof of Theorem \ref{easier}}
 
  For the  the converse implication, the key argument is that 
  the orbifold covering $\sY$ associated to the torsion free  normal finite index  subgroup 
 $ \Lam < \Ga : = \pi_1^{orb}(X),$
 is a locally symmetric manifold. 
 
Indeed we have :
 
 \begin{lemma}\label{orbifold-covering} (\cite[Lemma 3.1]{cat25})
 The orbifold $\sY$ is just a normal complex space, that is,  there are no marked divisors with  multiplicity $m_i \geq 2$.
 \end{lemma}

The above lemma  yields in particular that 
$\sY =Y := \sD/ \Lam$. 
Moreover, since $K_{\sX}$ is ample, then $K_Y=\pi^* (K_{\sX})$  is ample too.

 Case (i):
 if we assume that $Y$ is smooth, then  by  (ii) $Y$ admits a holomorphic curvature type tensor which satisfies conditions B) and C), and we may directly
 invoke Theorem \ref{manifold}. 

 Case (i'): we use the argument of  \cite{cat25} to show that $Y$ must be smooth.
 
 Let $Y'$ be a resolution of $Y$. Since by assumption  $\sR^1 \pi_* (\ZZ_{Y'})=0 $ 
 (this is true for instance if $Y$ has rational singularities) and 
 $\sR^2 \pi_* (\ZZ_{Y'})=0, $ we have an isomorphism
 $$ H^j(Y', \QQ) \cong H^j(Y, \QQ), j = 1,2.$$

 Hence  the degree $1$ morphism 
 $\al : Y' \ra Y$ yields an isomorphism of first and second cohomology groups.
 
 We follow a similar argument to the one used in \cite{nankai}, proof of Proposition 4.8: it suffices to show 
 that $\al$ is  finite, because then $\al$,  being finite and birational,  is  an isomorphism $ Y' \cong Y$ by normality,
 hence we have shown that $Y$ is smooth and we proceed as for case (i).

  Now, if  $\al$ 
  is not finite,  there is a curve  $C$ which is contracted 
  by $\al$, hence its homology class $c \in H_2(Y', \QQ)$ maps to zero in  $ H_2(Y, \QQ)$.
  And, since  $H^{2}(Y, \QQ) \cong H^{2}(Y', \QQ)$, the class $c$ of $C$, by the projection formula,  is orthogonal  to the whole
  of  $H^{2}(Y', \QQ)$, which is the pull back  of $H^{2}(Y, \QQ)$.

 This is a contradiction because, $Y'$ being projective, 
 the class $c$ of $C$ cannot be trivial.

 Therefore we can conclude  applying  Theorem \ref{manifold}.   
\hfill$\Box$ 
 
 \subsection{Proof of Theorem \ref{dm}}

 Let $\sY=Y$ be  the normal variety 
  associated to the torsion free  normal   subgroup of finite index
 $ \Lam \  \triangleleft  \ \Ga : = \pi_1^{orb}(\sX),$ and let $Y \to \sX$ be 
 the Galois covering induced by $G= \Gamma/\Lambda$. 
  Since $\sX$ is a Deligne-Mostov  orbifold,
  then $Y$ has quotient singularities and 
   we can consider the orbifold  K\"ahler-Einstein metric, which  exists and   is complete   by \cite[Thm. 1.3]{Fau} 
   and \cite{Yau93}. 

 By  assumption (2), $X$ admits an algebraic   curvature-type tensor of orbifold type,
hence, in view of Definition 
 \ref{orbifold type} and by  Lemma \ref{descent}, on $Y$ we obtain a holomorphic algebraic   curvature-type tensor
 $\sigma_Y \in H^0(End \big[ T_{Y^*}  \otimes  \Omega_{Y^*}^1\big])$.    By our hypotheses there is  
 a splitting of the tangent space $ T : = T_{Y,y} \cong T_{X,x}$ (here $y\mapsto x$) 
$$T = T'_1 \oplus  ... \oplus T'_m$$
such that the first Mok characteristic cone  $\sC\sS$  of $\s_Y$  is $\neq T$ and moreover $\sC\sS$ splits into m
irreducible components $\sC\sS' (j )$  such that properties B) and C) hold.

  Let   now $\tilde{Y} \to Y$ be its universal covering
and let $Y' \subset \tilde{Y}$ be its smooth part. 
 
By  \cite[Proposition 5.4]{campana},   $\tilde{Y}$ admits a De Rham orbifold  decompositions  
$\tilde{Y} = \prod_{i=1}^k Y_i,$  which restricts to a  decomposition on the smooth part 
$$ Y' = \prod_{i=1}^k Y'_i,$$
in such a way  that   the holonomy group decomposes as $\prod H_i$, where the action of each $H_i$  on each factor $Y'_i$ is irreducible and zero on 
the other components.

 Moreover,  
all the factors have holonomy different from the Unitary group and,  
since $K_Y$ is ample, there are no flat factors. Hence, by the Theorem of Berger \cite{Berger} and Simons 
the holonomy of each factor is the holonomy of an irreducible bounded symmetric domain. 

On the other hand,  the curvature type tensor $\s_Y$ 
 is parallel with respect to orbifold metric on   $Y$, by  a generalization of the standard Bochner  arguments  (see the book \cite{BoYa} and the paper  \cite{Koba80} for 
  a reference in  the smooth case,  
 \cite[Thm. 8.2]{GGK19} for a detailed analysis in the case of numerically  trivial canonical class and   \cite[Thm. B]{GrPa} for the case of ample canonical class), and 
 hence it induces a decomposition of the tangent bundle  $T_{Y'}$.  
Arguing as in \cite{CaDS2} we see that  $k=m$ and that 
the above De Rham decomposition  
corresponds to the  decomposition induced by the first Mok characteristic cone  $\sC\sS$  of $\s_{Y}$.

To conclude the proof we want to show that  $Y_i=Y_i' \cong \sD_i$, where $\sD_i$ is a bounded symmetric domain not of ball type. 

Indeed,    $Y'_i $  admits a holomorphic map $f'_i$ to an irreducible  bounded symmetric domain $\sD_i$ and 
by the Hartogs property, $f_i$ extends to a holomorphic map 
$$f_i :  {Y}_i \ra \sD_i,$$
which is an isometry when restricted to $Y'_i$.

Because of  completeness, $f_i$ is surjective, and since ${Y}_i $ is normal, its singular locus has codimension $2$, 
hence $ f'_i : Y'_i \ra f'_i (Y'_i)$ must be an isomorphism as the target is simply connected, and we have a covering space.

Again by Hartogs, the inverse of $f_i$, defined on $f'_i (Y'_i)$, extends to yield an isomorphism;
hence $f_i$ is an isomorphism.

In particular, it follows that $\tilde{Y}$  is smooth, hence we only need now to invoke Theorem \ref{manifold}.  
\hfill$\Box$

 \subsection{Proof of Theorem \ref{dm2}.}

Since the orbifold fundamental group $\Ga $ is residually finite, for each $D_i$, $i \in I = \{ 1, \dots, h \}$, there exists a finite quotient $G_i$ of
$\Ga$ such that the image of $\ga_i$ inside $G_i$ has order exactly $m_i$.

Let $G$ be the  image of $\Ga$ in $G_1 \times \dots G_h$.

Then the orbifold covering associated to $\Lambda : = ker (\Ga \ra G)$ is a reduced orbifold, and condition (3) is verified.
It suffices now to apply Theorem \ref{dm}.
\hfill$\Box$

 \subsection{Proof of Theorem \ref{theorem klt}.}

 Let $\sX =(X, D)$ be  a klt orbifold satisfying the hypotheses of Theorem \ref{theorem klt}.

  By  assumption (3)  there is  a  reduced klt orbifold 
 $Y$ and a 
 Galois orbifold cover with group $G$, $\pi: Y \to ( X , D)$.  In particular there exist open dense subsets $X^0\subset X_{reg}\setminus  \lceil D \rceil $ and 
 $Y^0\subset Y_{reg}$ such that $\pi: Y^0 \to X^0$  is \'etale. 
 
   Moreover, it is immediately seen   that  $K_Y=\pi^* (K_{\sX})$  is ample  since 
$K_{\sX}$ is ample by  our assumptions.   

   We  mimic the proof  of  \cite[Thm. 1.2]{CaDS2} extending it from   the case where $Y$ is smooth, 
to the case of a normal space $Y$ with  klt singularities.

By  assumption (2) $X$  admits a Mok curvature-type tensor of orbifold type
hence, in view of Definition 
 \ref{orbifold type} and by  Lemma \ref{descent}, on $Y$ we obtain a Mok curvature-type tensor
  $\sigma_Y$,   yielding a splitting of the tangent space at the general points
  (here as in the above theorem $y\mapsto x$) 
$$T : = T_{Y,y} \cong T_{X,x}  = T'_1 \oplus  ... \oplus T'_m.$$

  By the klt assumption, and by Theorem 7.12 of  \cite{Eys},  the pair $(X,D)$ admits a  singular 
 K\"ahler-Einstein metric in the  cohomology class of  $K_X + D$, regular on 
 $X_{reg}\setminus \lceil D\rceil$, whose pull back on $Y$   is  a K\"ahler-Einstein  metric which is 
 regular on the smooth locus $Y_{reg} $: this  will   be denoted $\omega_{Y_{reg} }$ .
Such a metric is not complete but we can  overcome this potential difficulty, using  \cite{discala}. 

As shown in Theorem B of \cite{GrPa} and Theorem 8.1 of \cite{GGK19}, we can use the Bochner  arguments to show that  the curvature type tensor is parallel with respect to this metric.

Let now  $\iota : Y_{reg} \ra Y$ be the inclusion of the smooth locus of $Y$, and set  
$ \s_{Y_{reg} } : =   \s_Y|_{Y_{reg} } $.

 We have   a holomorphic section  of 
 $End \big[ T_{Y_{reg} }  \otimes  \Omega_{Y_{reg} }^1\big] $, and 
by the Bochner principle  we can deduce that 
 the tensor $\s_{Y_{reg} }  $
  is parallel for the   orbifold  K\"ahler-Einstein metric  $\omega_{Y_{reg} }$. 
  
Hence, if we restrict the tensor $\s_{Y_{reg} }  $ at any smooth point $ y \in Y_{reg}$ we observe that $\s_{Y_{reg} }  $ 
is   invariant for the restricted Holonomy group $H^0$.

 Moreover    the above  decomposition of the tangent sheaf and the holonomy invariance 
   induce (locally)  a corresponding decomposition $Y=Y_1 \times \cdots Y_m$   without 
  non-trivial flat factors since $Ric(\omega_{Y_{reg} }) = - \omega_{Y_{reg} }$  is negative definite, since by our assumption   $K_Y$ is ample.

By \cite{discala}, and since the metric is locally isometric to the one of a bounded symmetric domain without flat factors, the restricted Holonomy group $H^0$ is of finite index in its Normalizer inside the orthogonal group: a fortiori
$H^0$ is of finite index  in the full Holonomy group $H$.

Hence, passing to another (finite)  Galois 
orbifold covering $Y' \ra Y$ ($Y'$  is again a klt normal complex space), we may assume that we have a reduced
orbifold covering $Y'$ where  $\s_{Y'_{reg} }  $ is invariant for the full holonomy group, which equals the restricted
holonomy group \footnote{ Note that this step had been  done   earlier in \cite[Thm. 1.14 and Thm. 3.7]{GKP16}, where  the authors
obtain an orbifold covering, 
 and call it   the Galois universal closure  $\eta\colon Y' \to Y$. }.

Let $\sT_{Y'}$  be the tangent sheaf of $Y'$, which can be defined  as $  \iota_* \sT{_{Y'_{reg} }}$  (cf. \cite{GKPT}).

The  Mok tensor $\s$ induces a holonomy invariant decomposition of the holomorphic Tangent sheaf  $\sT{_{Y'_{reg} }}.$

While on $Y_{reg}$  we could write only 
locally    the identity
as a sum of idempotents $\phi_i$, since on $Y'$ 
the individual summands (which we shall call Mok subbundles) are preserved by the holonomy group,
we can write globally on $Y'_{reg} $ the identity
as a sum of such idempotents $\phi_i$.

Since on  $Y'$ we have global idempotents   $\phi_i$ and these, by the Hartogs property,     extend to  idempotents  
$\hat{\phi_i} \in End ( \Omega^{[1]}_{Y'})$ 
such that the identity equals their sum, 
defining $\sF_i : = \ker (1 - \hat{\phi_i})$, the Mok subbundles extend  to subsheaves on $Y'$ and we can write 
$$ \Omega^{[1]}_{Y'} = \oplus_i \sF_i .$$

We have a similar decomposition of $\sT_{Y'}$.

The Holonomy invariance  of such a decomposition
corresponds to the existence of  a 
 {\bf  uniformizing variation of Hodge structure} and hence to the existence of a Higgs  sheaf  $\sE$  on $ Y'_{reg}$
 which contains $T_{Y'_{reg}}$ as a direct summand 
 (see Propositions 6.3  and 6.6  of  \cite{GrPa2} for a detailed explanation, relating the holonomy groups $H_i$ with the
 Hodge groups of the decomposition of the Higgs bundle as a sum of polystable  sheaves).

  By \cite[Thm. 1.14]{GKP16}  $\sE$ extends to a locally free sheaf on the entire   $Y'$ and hence also the tangent sheaf  $\sT_{Y'}$, being a direct summand of a locally free sheaf is locally free  
  (see  also \S 5 and in particular Corollary 5.39 of  \cite{GKPT}  for these kind of arguments).

  Therefore,  by the Lipman-Zariski conjecture, valid in dimension $n \geq 3$,  the klt space $Y'$  is a smooth manifold
  if it has dimension $\geq 3$.  
  
  The case of dimension $n \leq 2$ is indeed well known (see \cite{Yau93}, and also for instance \cite{CaFr}).

To conclude we only need now to apply Theorem \ref{manifold} to $Y'$ and to consider the composition map $Y' \to Y \to (X,D)$, which is an orbifold covering. 
\hfill$\Box$

\begin{remark}
P. Graf and A. Patel  (Theorem 1.3 of  \cite{GrPa2}) prove  Theorem 4 
for the case of  klt varieties (that is, for reduced klt orbifolds) .
\end{remark} 

 \bigskip
 
{\bf Acknowledgements:} we thank  Antonio Di Scala for making us aware of his recent article \cite{discala}. 
Marco Franciosi   was    supported by the project PRIN 
 2022BTA242 ``Geometry of algebraic structures: Moduli, Invariants, Deformations''
  of Italian MUR,  and member of GNSAGA of INDAM.


\end{document}